\documentclass[11pt]{article}
\usepackage{amsmath,amsfonts,amssymb,amsthm} 
\usepackage[a4paper,textwidth=145mm,textheight=220mm]{geometry}
\newtheorem{remark}{Remark}[section]

\newtheorem{theorem}[remark]{Theorem}
\newtheorem{proposition}[remark]{Proposition}

\newtheorem{lemma}[remark]{Lemma}

\newcommand{\frakH}{\mathfrak{H}}
\newcommand{\frakK}{\mathfrak{K}}

\newcommand{\frakr}{\mathfrak{r}}
\newcommand{\fraks}{\mathfrak{s}}
\newcommand{\frakt}{\mathfrak{t}}
\newcommand{\frakm}{\mathfrak{m}}
\newcommand{\fraki}{\mathfrak{i}}
\DeclareMathOperator{\clus}{clust}
\DeclareMathOperator{\Ad}{Ad}
\DeclareMathOperator{\supp}{supp}
\newcommand{\GsrG}{G {_{s}\times_{r}} G}

\title{ The  Fell compactification and non-Hausdorff groupoids}

\author{Thomas Timmermann\\[1ex]
\texttt{timmermt@math.uni-muenster.de}\\ SFB 478
``Geometrische Strukturen in der Mathematik''\\ Hittorfstr.\
27, 48149 M\"unster}

\date{\today}

\begin{document}
\maketitle

\begin{abstract}
 A compactification of Fell is applied to locally compact
  non-Hausdorff groupoids and yields locally compact
  Hausdorff groupoids. In the \'etale case, this
  construction provides a geometric picture for the
  left-regular representations introduced by Khoshkam and
  Skandalis.

\end{abstract}

\section{Introduction}

In \cite{fell}, Fell introduced a compactification of
locally compact non-Hausdorff spaces. We apply this
construction to groupoids with open range and source map and
Hausdorff unit space and show that if one leaves out the
point at infinity, one obtains locally compact Hausdorff
groupoids. If the non-Hausdorff  groupoid   was \'etale, then also
the Hausdorff groupoid is \'etale and provides a geometric
picture for the construction of the left-regular
representation of the non-Hausdorff groupoid given by
Khoshkam and Skandalis in \cite{khoshkam}.
We also comment on related constructions of Tu \cite{tu}.
In the remainder of the introduction, we recall the
constructions and results of \cite{fell}.

\paragraph{The Fell compactification}
Let $X$ be a topological space. We call a subset $K
\subseteq X$ {\em quasi-compact} if it has the finite
covering property, and {\em compact} if it is quasi-compact
and Hausdorff.\footnote{ In \cite{fell}, Fell calls compact
  what we call quasi-compact.}  We assume that $X$ is {\em
  locally compact} in the sense that every point $x \in X$
has a compact neighbourhood. In that case, every
neighbourhood of any point $x \in X$ contains a compact
neighbourhood of $x$, and $X$ is locally compact in the
sense of \cite{fell}.

Let $(x_{\nu})_{\nu}$ be a net in $X$. We denote by
$\clus_{\nu} x_{\nu}$ and $\lim_{\nu} x_{\nu}$ the set of
all cluster points and the set of all limit points of this
net, respectively, and call it {\em primitive} if it
satisfies the following equivalent conditions:
  \begin{enumerate}
  \item  Every cluster point of the net is a limit point:
    $\clus_{\nu} x_{\nu} = \lim_{\nu} x_{\nu}$.
  \item Each $x \in X \setminus \lim_{\nu} x_{\nu}$ has a
    neighbourhood that is eventually left by the net.
  \item Every quasi-compact set that contains no limit point
    is eventually left by the net. 
  \end{enumerate}

  The {\em Fell compactification} of $X$ is the set $\frakK
  X$ of all limit sets of primitive nets in $X$, equipped
  with the topology induced by the subbasis that consists of all sets of the
  form ${\cal U}_{V} = \{ A \in \frakK X \mid A \cap V \neq
  \emptyset\}$ and ${\cal U}^{Q} = \{ A \in \frakK X \mid A
  \cap Q =\emptyset \}$, where $V \subseteq X$ is open and
  $Q \subseteq X$ is quasi-compact.  The space $\frakK X$ is
  compact \cite{fell}, and  $X$ embeds into $\frakK
  X$ since each constant net is primitive. This embedding is
  not necessarily continuous but has dense image.  For each
  net $(x_{\nu})_{\nu}$ in $X$, the net
  $(\{x_{\nu}\})_{\nu}$ converges in $\frakK X$ to some set
  $A \subseteq X$ if and only if $(x_{\nu})_{\nu}$ is
  primitive and $A=\lim_{\nu} x_{\nu}$ in $X$ \cite{fell}.
  If $X$ is quasi-compact, then $\emptyset \not\in \frakK X$
  and we let $\frakH X := \frakK X$. Otherwise, $\emptyset
  \in \frakK X$ and $\frakK X$ is the one-point
  compactification of $\frakH X:=\frakK X \setminus \{
  \emptyset\}$.  If $X$ is Hausdorff, then $\frakH X = X$.

  The Fell compactification is functorial in the following
  sense.  Let $X,Y$ be locally compact spaces and let $f
  \colon X \to Y$ be continuous and {\em proper} in the
  sense that the preimage of every quasi-compact subset is
  quasi-compact again. Using condition 3 above, one finds
  that the image of each primitive net in $X$ is primitive
  in $Y$, so that the map $\frakK f \colon \frakK X \to
  \frakK Y$, $A \mapsto f(A)$, is well-defined. One easily
  checks that this map is continuous and that the
  restriction $\frakH f :=\frakK f |_{\frakH X} \colon
  \frakH X \to \frakH Y$ is proper.

\paragraph{Quasi-continuous functions}

Let $X$ be as above and let $Y$ be a locally compact
Hausdorff space. A map $f\colon X \to Y$ is {\em
  (w-)quasi-continuous} if for each primitive net
$(x_{\nu})_{\nu}$ in $X$ (with non-empty limit set), the net
$(f(x_{\nu}))_{\nu}$ converges in $Y$. Evidently, the
restriction $f|_{X}$ of each continuous map $f \colon \frakH
X \to Y$ is w-quasi-continuous. Conversely, the following
lemma implies that for each w-quasi-continuous function $f
\colon X \to Y$, the extension $\tilde f \colon \frakH X \to Y$ defined by
$\tilde f(\lim_{\nu} x_{\nu})=\lim_{\nu}f(x_{\nu})$ for each
primitive net $(x_{\nu})_{\nu}$ in $X$ is continuous.
 \begin{lemma} \label{lemma:continuous} Let $Z$ be a
   topological space with a dense subset $Z_{0}$. A map $f
   \colon Z \to Y$ is continuous if and only if
   $f(z)=\lim_{\nu} f(z_{\nu})$ for each limit point $z\in
   Z$ of each net $(z_{\nu})_{\nu}$ in $Z_{0}$.
\end{lemma}
\begin{proof}
  Let $f \colon Z \to Y$ be some map, let
  $(z_{\nu})_{\nu}$ be a net in $Z$ with limit point $z \in
  Z$, and assume that $(f(z_{\nu}))_{\nu}$ does not converge
  to $f(z)$. Replacing $Y$ by its one-point compactification
  $Y^{+}$ and $(z_{\nu})_{\nu}$ by a subnet, we may assume
  that $(f(z_{\nu}))_{\nu}$ converges to some $y \in
  Y^{+}$. Choose disjoint neighbourhoods $V_{y}, V_{f(z)}$
  of $y$ and $f(z)$. For each neighbourhood $U$ of $z$, we
  find some $\nu$ such that $z_{\nu} \in U$ and, using the
  assumption on $(z_{\nu})_{\nu}$, a point $z_{U} \in U \cap Z_{0}$
  such that $f(z_{U}) \in V_{y}$. Ordering the
  neighbourhoods of $z$ by inclusion, we obtain a net
  $(z_{U})_{U}$ in $Z_{0}$ that converges to $z$ and such
  that $f(z_{U}) \not\in V_{f(z)}$ for all $U$. This is a contradiction.
\end{proof} 
Summarizing, we obtain a bijection between all
w-quasi-continuous maps $X \to Y$ and all continuous maps
$\frakH X \to Y$, as already stated in
\cite{fell} but without proof.

\section{Application to locally compact
  groupoids}

Let $G$ be a locally compact groupoid such that the unit
space $G^{0}$ is Hausdorff and the range and the source maps
$r,s \colon G \to G^{0}$ are open
\cite{paterson,renault}. Recall that an {\em action} of $G$
on a topological space $X$ consists of continuous maps $\rho
\colon X \to G^{0}$ and $\mu \colon G {_{s}\times_{\rho}} X
\to X$, written $(x,y) \mapsto xy$, such that $\rho(x)x =
x$, $\rho(\gamma x)=r(\gamma)$, and $
(\gamma'\gamma)x=\gamma'(\gamma x)$ for all $x\in X$,
$\gamma \in G_{\rho(x)}$, $\gamma' \in G_{r(\gamma)}$.
\begin{proposition} \label{proposition:action}
  Let $(\rho,\mu)$ be an action of $G$ on a locally compact
  space $X$, where $\rho$ is w-quasi-continuous. Then
  $(\rho,\mu)$ extends to an action $(\tilde \rho,\tilde
  \mu)$ of $G$ on $\frakH X$.
\end{proposition}
\begin{proof}
  By assumption, $\rho$ extends to a continuous map $\tilde
  \rho \colon \frakH X \to G^{0}$, and we only need to show
  that the formula $(\gamma,A) \mapsto \gamma A$ defines a
  continuous map $\tilde \mu \colon G {_{s}\times_{\tilde
      \rho}} \frakH X \to \frakH X$.  Let $(\gamma,A) \in G
  {_{s}\times_{\tilde \rho}} \frakH X$, where $A$ is the
  limit set of a primitive net $(x_{\nu})_{\nu}$ in $X$.
  Denote by $I$ the set of pairs $(U,\nu)$ such that $U$ is
  a neighbourhood of $\gamma$ and $\tilde \rho(x_{\nu'}) \in
  s(U)$ whenever $\nu' \geq \nu$, and equip $I$ with an
  order such that $(U,\nu) \geq (U',\nu')$ if and only if $U
  \subseteq U'$ and $\nu \geq \nu'$. For each $(U,\nu) \in
  I$, choose $\gamma_{(U,\nu)} \in U$ such that
  $s(\gamma_{(U,\nu)}) = \rho(x_{\nu})$. Then
  $\gamma=\lim_{(U,\nu)} \gamma_{(U,\nu)}$ and hence $\gamma
  A \subseteq \lim_{(U,\nu)} \gamma_{(U,\nu)}x_{\nu}$. If
  $x' \in \clus_{(U,\nu)} \gamma_{(U,\nu)}x_{\nu}$, then
  $\rho(x') = \lim_{\nu} r(\gamma_{(U,\nu)}) = r(\gamma)$
  and $\gamma^{-1}x' \in \clus_{(U,\nu)}
  \gamma_{(U,\nu)}^{-1}\gamma_{(U,\nu)} x_{\nu} =
  \clus_{\nu} x_{\nu} = A$, whence $x' \in \gamma A$. Thus,
  the net $(\gamma_{(U,\nu)}x_{\nu})_{(U,\nu)}$ is primitive
  and $\gamma A \in \frakH X$.  Therefore, the map $\tilde
  \mu$ is well defined. The argument above shows that the
  subset $G {_{s}\times_{\rho} X} \subseteq G
  {_{s}\times_{\tilde \rho}} \frakH X$ is dense. If
  $((\gamma_{\nu},x_{\nu}))_{\nu}$ is a net in $G
  {_{s}\times_{\rho} X}$ that converges to some $(\gamma,A)
  \in G {_{s}\times_{\tilde \rho}} \frakH X$, then the net
  $(x_{\nu})_{\nu}$ is primitive and a similar argument as
  above shows that the net $(\gamma_{\nu}x_{\nu})_{\nu}$ is
  primitive and converges to $\gamma A$. By Lemma
  \ref{lemma:continuous}, $\tilde \mu$ is continuous.
\end{proof}

\begin{theorem} \label{theorem:groupoid} The space $\frakH
  G$ carries the structure of a locally compact Hausdorff
  groupoid  such that the unit
  space $(\frakH G)^{0}$ is the closure of $G^{0}$ in
  $\frakH G$ and the range map $\frakr$, the source map
  $\fraks$, the multiplication $\frakm$ and the inversion
  $\fraki$ are given by
  \begin{align*}
    \frakr(A) &= AA^{-1}, & \fraks( A)
    &= A^{-1}A, & \frakm((A,B)) &= AB, &
    \fraki(A) &= A^{-1} &\text{for all } A,B \in
    \frakH G.
  \end{align*}
\end{theorem}
The proof involves several lemmas:
\begin{lemma} \label{lemma:multiplication}
  Let $((x_{\nu},y_{\nu}))_{\nu}$ be a net in $\GsrG$ such that
  $(x_{\nu})_{\nu}$ and $(y_{\nu})_{\nu}$ are primitive with
  limit sets $A,B \in \frakH G$. Then the net
  $(x_{\nu}y_{\nu})_{\nu}$ is primitive with limit set $AB$.
\end{lemma}
\begin{proof}
  Clearly, $\clus_{\nu} x_{\nu}y_{\nu} \subseteq
  AA^{-1}\clus_{\nu} x_{\nu}y_{\nu} \subseteq A \clus_{\nu}
  x^{-1}_{\nu}x_{\nu}y_{\nu} =A \clus_{\nu} y_{\nu} = AB
  \subseteq \lim_{\nu} x_{\nu} y_{\nu}$.
\end{proof}
\begin{lemma} \label{lemma:subgroup}
  \begin{enumerate}
  \item $A=AA^{-1}A$ and $A=xA^{-1}A$ for each $A \in \frakH
    G$ and $x \in A$.
  \item $AB=xyB^{-1}B$ for all $A,B \in \frakH G$, $x \in A,
    y \in B$ satisfying $A^{-1}A=BB^{-1}$.
  \end{enumerate}
\end{lemma}
\begin{proof}
1. If $A=\lim_{\nu} x_{\nu}$ for a primitive net
$(x_{\nu})_{\nu}$ in $G$, then
  $AA^{-1}A = \lim_{\nu} x_{\nu}x_{\nu}^{-1}x_{\nu} =
  \lim_{\nu} x_{\nu} = A$ by the  lemma above, and if $x,y \in
  A$, then $y =xx^{-1}y \in xA^{-1}A$.

2. By 1., $AB=xA^{-1}AB= xBB^{-1}B=xyB^{-1}B$.
\end{proof}
\begin{lemma} \label{lemma:w-quasi}
  The range and the source map of $G$ are w-quasi-continuous.
\end{lemma}
\begin{proof}
  Let $(x_{\nu})_{\nu}$ be a primitive net in $G$ with
  non-empty limit set. By Lemma \ref{lemma:multiplication},
  the nets $(r(x_{\nu}))_{\nu} =
  (x_{\nu}x_{\nu}^{-1})_{\nu}$ and $(s(x_{\nu}))_{\nu} =
  (x_{\nu}^{-1}x_{\nu})_{\nu}$ are primitive in $G$ and
  therefore also in $G^{0}$, and they have limit points in
  $G^{0}$ because $r,s$ are continuous.
\end{proof}
The range map $r \colon G \to G^{0}$ and the multiplication
map $m \colon \GsrG \to G$ form an action on $G$. By
Lemma \ref{lemma:w-quasi} and Proposition
\ref{proposition:action}, it extends to an action
$(\tilde r,\tilde m)$ on $\frakH G$.
 \begin{proof}[Proof of Theorem \ref{theorem:groupoid}]
   By Lemma \ref{lemma:multiplication} and
   \ref{lemma:continuous}, the maps $\frakr,\fraks,\fraki$
   are well defined and continuous. The map $\frakm$ is well
   defined because $AB=xB = \tilde m(x,B) \in \frakH G$ for
   all $(A,B) \in \frakH G {_{\fraks}\times_{\frakr}} \frakH
   G$ and $x \in A$. Equipped with these maps,
   $\frakH G$ becomes a groupoid, as one can easily check
   using Lemma \ref{lemma:subgroup}. We show that
   the map $\frakm$ is continuous.  Assume that
   $((A_{\nu},B_{\nu}))_{\nu}$ is a net in $\frakH G
   {_{\fraks}\times_{\frakr}} \frakH G$ that converges to
   some point $(A,B)$. Choose $x \in A$ and $y \in B$.  If
   $U,V$ are neighbourhoods of $x,y$, then ${\cal
     U}_{U},{\cal U}_{V}$ are neighbourhoods of $A,B$ and
   hence there exists some $\nu$ such that $A_{\nu} \in
   {\cal U}_{U}$, $B_{\nu} \in {\cal U}_{V}$, that is,
   $A_{\nu} \cap U \neq \emptyset$, $B_{\nu} \cap V \neq
   \emptyset$.  Denote by $I$ the set of all triples
   $(U,V,\nu)$, where $U,V$ are neighbourhoods of $x,y$ and
   $A_{\nu} \cap U \neq 0$, $B_{\nu} \cap V \neq 0$. Order
   $I$ such that $(U,V,\nu) \geq (U',V',\nu')$ if and only
   if $U \subseteq U'$, $V \subseteq V'$, $\nu \geq \nu'$. Let
   $(U,V,\nu) \in I$ and choose $x_{\nu}^{U} \in A_{\nu}
   \cap U$, $y_{\nu}^{V} \in B_{\nu} \cap V$.  The nets
   $(x_{\nu}^{U}y_{\nu}^{V})_{(U,V,\nu)}$ and
   $(B_{\nu}^{-1}B_{\nu})_{(U,V,\nu)}$ converge to $xy$ and
   $B^{-1}B$, respectively, and hence $A_{\nu}B_{\nu} =
   \tilde m(x_{\nu}^{U}y_{\nu}^{V}, B_{\nu}^{-1}B_{\nu})$
   converges to $\tilde m(xy,B^{-1}B) = AB$ in $\frakH G$.
\end{proof}
We call an open subset $V \subseteq G$ a {\em $G$-set} if
$r(V) \subseteq G^{0}$ is open and $r|_{V} \colon V \to
r(V)$ is a homeomorphism. Recall that $G$ is \'etale if it
is covered by its open $G$-sets.
\begin{proposition}
  If $G$ is \'etale, then $\frakH G$ is \'etale.
\end{proposition}
\begin{proof}
  Let $V \subseteq G$ be an open Hausdorff $G$-set and let
  $s_{V}:=s|_{V} \colon V \to s(V)$. We show that
  $\fraks({\cal U}_{V})={\cal U}_{s(V)} \cap (\frakH G)^{0}$
  and construct a continuous map $\frakt \colon {\cal
    U}_{s(V)} \cap (\frakH G)^{0} \to {\cal U}_{V}$ that is
  inverse to $\fraks$, and then the claim follows because
  sets of the form ${\cal U}_{V}$ cover $\frakH G$. First,
  $\fraks({\cal U}_{V}) \subseteq {\cal U}_{s(V)} \cap
  (\frakH G)^{0}$ because $A \cap V \neq \emptyset$ implies
  that $A^{-1}A \cap s(V) \neq \emptyset$ for each $A \in
  \frakH G$. Clearly, $\tilde s(B) \in s(V)$ for each $B \in
  {\cal U}_{s(V)}$.  The map $\frakt \colon {\cal U}_{s(V)}
  \cap (\frakH G)^{0} \to \frakH G$ given by $B \mapsto
  s_{V}^{-1}(\tilde s(B))B = \tilde m(s_{V}^{-1}(\tilde
  s(B)),B)$ is continuous because the maps
  $s_{V}^{-1},\tilde s,\tilde m$ are continuous, its image
  is contained in ${\cal U}_{V}$ because $s_{V}^{-1}(\tilde
  s(B)) \in s_{V}^{-1}(\tilde s(B)) B$ for each $B \in {\cal
    U}_{s(V)}$, and $\frakt(\fraks(A)) = (A \cap V) A^{-1}A
  = A$ and $\fraks(\frakt(B)) = B^{-1}B = B$ for each $A \in
  {\cal U}_{V}$, $B \in {\cal U}_{s(V)} \cap (\frakH
  G)^{0}$.
\end{proof}
The groupoid $\frakH G$ can be identified with the quotient
of the transformation groupoid \cite{paterson} associated to an adjoint
action of $G$ on $(\frakH G)^{0}$ as follows.
\begin{proposition} \label{proposition:transformation}
  \begin{enumerate}\item 
    There exists an action $(\tilde t,\tilde \Ad)$ of $G$ on
    $(\frakH G)^{0}$ such that $\tilde t(A) =\tilde
    r(A)=\tilde s(A)$ and $\tilde \Ad(x,A) = xAx^{-1}$ for
    all $A \in (\frakH G)^{0}$, $x \in G_{\tilde t(A)}$.
  \item The map $G \ltimes (\frakH G)^{0} \to
    \frakH G$ given by $(x,A) \mapsto xA$ is a continuous
    surjective groupoid homomorphism with kernel $N= \{(x,A)
    \in G \ltimes (\frakH G)^{0} \mid x\in A\}$.
  \item The induced map $(G \ltimes (\frakH G)^{0})/N \to
    \frakH G$ is an isomorphism of topological groupoids.
  \end{enumerate}
\end{proposition}
\begin{proof}
  1. The only nontrivial assertion is continuity of the map
  $\tilde \Ad$. Since $s$ is open, $\GsrG^{0}$ is dense in
  $G {_{s}\rtimes_{\tilde t}} (\frakH G)^{0}$ (see the proof
  of Proposition \ref{proposition:action}), and continuity
  follows from Lemma \ref{lemma:continuous} and
  \ref{lemma:multiplication}.
  
  2. The map is a groupoid homomorphism because for each
  pair of composable elements $(x,A),(x',A')) \in G \ltimes
  (\frakH G)^{0}$, we have $\fraks(xA)=A^{-1}x^{-1}xA =
A$, $\frakr(x'A') =
  x'A'A'{}^{-1}x'{}^{-1}=x'A'x'{}^{-1}$, and
  $xAx'A'=xx'A'$. It is continuous because it is the
  restriction of $\tilde m$, and surjective because $A =
  xA^{-1}A$ for each $A \in \frakH G$ and $x \in A$.

  3. Denote by $q$ the inverse of the induced map. We have
  to show that $q$ is continuous and do so using Lemma
  \ref{lemma:continuous}. If $(x_{\nu})_{\nu}$ is a
  primitive net in $G$ with limit set $A$ and $x \in
  A$, then the net $(q(x_{\nu}))_{\nu} =
  ((x_{\nu},s(x_{\nu}))N)_{\nu}$ clearly converges to
  $q(A)=(x,A^{-1}A)N$.
\end{proof}
\begin{proposition}
  Let $(\rho,\mu)$ be an action of $G$ on a locally compact
  space $X$, where $\rho$ is quasi-continuous. Then
  $(\rho,\mu)$ extends to an action $(\hat \rho,\hat \mu)$
  of $\frakH G$ on $\frakH X$ such that $\hat
  \rho(B)=\rho(B)$ and $\hat \mu(A,B)=AB$ for all $A \in
  \frakH G$, $B \in \frakH X$.
\end{proposition}
\begin{proof}
  For each primitive net $(x_{\nu})_{\nu}$ in $X$, the net
  $(\rho(x_{\nu}))_{\nu}$ converges in $G^{0}$ and therefore
  is primitive in $G$ and converges in $\frakH
  G$. Hence, $\rho$ is w-quasi-continuous as a map to
  $\frakH G$ and extends as claimed. Similar arguments as in
  the proof of Theorem \ref{theorem:groupoid} show that the
  map $\hat \mu$ is well defined and continuous.
  Clearly, $(\hat \rho,\hat \mu)$ is an action.
\end{proof}
By functoriality of $\frakH$, the assignment $G \mapsto
\frakH G$ extends to a functor from the category of locally
compact groupoids with open range and source maps and
Hausdorff unit space together with proper and continuous
groupoid homomorphisms into the category of locally compact
Hausdorff groupoids with proper continuous homomorphisms.

\section{Relations to other constructions}

\paragraph{A construction of Tu}
Let $X$ be a locally compact space.  In \cite{tu},
Jean-Louis Tu associated to $X$ a Hausdorff space ${\cal H}
X$ as follows. As a set, ${\cal H} X$ consists of all
subsets $A \subset X$ satisfying the following condition:
for every family $(V_{x})_{x \in A}$ of open sets such that
$x \in V_{x}$ for all $x \in A$ and $V_{x}=X$ except perhaps
for finitely many $x \in A$, one has $\bigcap_{x \in A}
V_{x} \neq \emptyset$. This set is endowed with the topology
generated by all subsets of the form $\Omega_{V} = \{ A \in
{\cal H}X \mid A\cap V\neq\emptyset\}$ and $\Omega^{Q} =\{
A\in {\cal H}X\mid A \cap Q=\emptyset\}$, where $V \subseteq
X$ is open and $Q \subseteq X$ is quasi-compact.
\begin{proposition}
  $\frakH X$ is a subspace of ${\cal H}X$.
\end{proposition}
\begin{proof}
  Let $A \in \frakH X$ and let $(V_{x})_{x \in A}$ be a
  family of open sets such that $x \in V_{x}$ for all $x \in
  A$ and $V_{x}=X$ for all but
  finitely many $V_{1},\ldots,V_{n}$. Then ${\cal
    U}_{V_{1}} \cap \cdots \cap {\cal U}_{V_{n}}$ is a
  neighbourhood of $A$ and therefore contains a point $x \in
  X$, that is, $V_{1} \cap \cdots \cap V_{n} \neq
  \emptyset$.  The topology on $\frakH X$
  coincides with the subspace topology inherited from ${\cal
    H}X$ by definition.
\end{proof}
Note that Proposition \ref{proposition:action} is an
analogue of \cite[Proposition 3.10]{tu}.

\paragraph{A spectral picture of the Fell compactification}
Let $X$ be a locally compact space and denote by ${\cal
  B}(X)$ the $C^{*}$-algebra of all bounded Borel functions
on $X$. Given an open Hausdorff subset $U \subseteq X$ and a
function $f \in C_{c}(U)$, we identify $f$ with a function
on $X$ extending it outside of $U$ by $0$, and denote by
$\supp f$ the support of $f$ inside $U$. Denote by ${\cal A}_{c}(X)
\subseteq {\cal B}(X)$ the smallest subalgebra containing
$C_{c}(U)$ for each open Hausdorff subset $U \subseteq X$,
and by ${\cal A}_{0}(X)$ its norm closure.
\begin{proposition} \label{proposition:spectral}
  Restriction of functions defines a $*$-isomorphism
  $C_{0}(\frakH X) \to {\cal A}_{0}(X)$ that maps
  $C_{c}(\frakH X)$ to ${\cal A}_{c}(X)$.
\end{proposition} 
\begin{proof}
  Let $U \subseteq X$ be open and Hausdorff and let $f \in
  C_{c}(U)$.  If $(x_{\nu})_{\nu}$ is a primitive net in $X$
  with non-empty limit set, then either $\supp f$ contains a
  limit point of the net, $x$, say, or $\supp f$ is eventually left by
  the net. In the first case, $f(x_{\nu})$ converges to $f(x)$
  because $f$ is continuous on $U$, and in the second case,
  $f(x_{\nu})$ is constant $0$. Thus, $f$ is
  w-quasi-continuous on $X$. Consequently, each function in
  ${\cal A}_{0}(X)$ is w-quasi-continuous and extension of
  functions defines an embedding $\iota \colon {\cal
    A}_{0}(X) \hookrightarrow C_{0}(\frakH X)$.  If $A,B \in
  \frakH X$ and $x \in A \setminus B$, then there exists an
  open Hausdorff neighbourhood $U$ of $x$ that is disjoint
  to $B$, and for each $f \in C_{c}(U)$ we have
  $(\iota(f))(A)=f(x)$ and $(\iota(f))(B)=0$.  Therefore,
  $\iota({\cal A}_{c}(X))$ separates the points of $\frakH
  X$.  By the Stone-Weierstrass theorem, $\iota$ is
  surjective, and it follows that $C_{c}(\frakH X) \subseteq
  \iota({\cal A}_{c})$. The last inclusion is an equality
  because for each $f \in {\cal A}_{c}(X)$, there exists a
  quasi-compact set $Q \subseteq X$ such that $f|_{X
    \setminus Q} = 0$, and ${\cal U}^{Q}$ is a neighbourhood
of $\emptyset$  in $\frakK X$  such that $\iota(f)|_{{\cal
      U}^{Q} \cap \frakH X } \equiv 0$.
\end{proof}

\paragraph{The left-regular representation of Khoshkam and Skandalis}
Let $G$ be an \'etale, locally compact groupoid with open
range map and Hausdorff unit space.  In \cite{khoshkam},
Mahmood Khoshkam and Georges Skandalis define a left-regular
representation of $G$ using Hilbert $C^{*}$-modules
\cite{lance} and functional-analytic tools.  The action $\tilde m \colon G
{_{s}\times_{\tilde r}} \frakH G \to \frakH G$ provides the
following geometric picture for their construction.

Let $C_{c}(G)$ be the linear span of all subspaces
$C_{c}(U)$, where $U \subseteq G$ is open and
Hausdorff. Then $C_{c}(G)$ is a $*$-algebra with respect to
the operations
\begin{align} \label{eq:convolution} (f \ast g)(x) &=
  \sum_{z \in G^{r(x)}} f(z)g(z^{-1}x), & f^{*}(x) &=
  \overline{f(x^{-1})}, && \text{ where } x \in G, f,g \in
  C_{c}(G).
\end{align}
Denote ${\cal B}(G^{0})$ the $C^{*}$-algebra of bounded
Borel functions. Let $D \subseteq {\cal B}(G^{0})$ be the
$C^{*}$-subalgebra generated by all restrictions
$f|_{G^{0}}$, where $f \in C_{c}(G)$, let $Y$ be the
spectrum of $D$, and identify $D$ with $C_{0}(Y)$ in the
canonical way. Then the algebraic tensor product $C_{c}(G)
\odot C_{0}(Y)$ is a pre-Hilbert $C^{*}$-module over
$C_{0}(Y)$ with respect to the inner product and right
module structure given by
\begin{align*}
  \langle f \odot a|g \odot b\rangle &= a^{*} \cdot
  (f^{*} \ast g)|_{G^{0}} \cdot b, & (f \odot a)b &= f \odot ab,
  && \text{where } f,g \in C_{c}(G), a,b \in
  C_{0}(Y).
\end{align*}
Denote by $L^{2}(G)$ the separated completion.   There exists a
representation $L_{G}\colon C_{c}(G) \to {\cal L}(L^{2}(G))$
given by $L_{G}(f)(g \odot a) = f \ast g \odot a$ for all
$f,g \in C_{c}(G)$, $a \in C_{0}(Y)$. Denote by
$C^{*}_{r}(G)$ the completion of
$L_{G}(C_{c}(G))$. Proposition \ref{proposition:spectral}
applied to $G$ immediately implies:
\begin{lemma} \label{lemma:y} The map $f \mapsto f|_{G^{0}}$
  is a $*$-isomorphism $C_{0}((\frakH G)^{0}) \to C_{0}(Y)
  \subseteq {\cal B}(G^{0})$.  In particular, $Y$ is
  homeomorphic to $(\frakH G)^{0}$.  \qed
\end{lemma}
The space $C_{c}(\frakH G)$ carries the structure of a
$*$-algebra, where the operations are defined similarly as
in \eqref{eq:convolution}, and the structure of a
pre-Hilbert $C^{*}$-module over $C_{0}((\frakH G)^{0})$,
where $\langle f|g\rangle = (f^{*} \ast g)|_{(\frakH G)^{0}}$ and
$(fh)(A) = f(A)h(\fraks(A))$ for all $f,g \in C_{c}(\frakH
G), h \in C_{0}((\frakH G)^{0}), A \in \frakH G$.  Denote
the completion of this pre-Hilbert $C^{*}$-module by
$L^{2}(\frakH G)$. Then there exists a representation
$L_{\frakH G} \colon C_{c}(\frakH G) \to {\cal
  L}(L^{2}(\frakH G))$ such that $L_{\frakH G}(f)g=f \ast g$ for
all $f,g \in C_{c}(G)$, and the closure $C^{*}_{r}(\frakH
G)=\overline{L_{\frakH G}(C_{c}(G))}$ is the reduced
$C^{*}$-algebra of $\frakH G$ \cite{paterson,renault}.  Identify $C_{0}(Y)$
with $C_{0}((\frakH G)^{0})$ as in Lemma \ref{lemma:y}, and
$C_{c}(G)$ with a subspace of ${\cal A}_{c}(G)=C_{c}(\frakH
G)$, see Proposition
\ref{proposition:spectral}. Denote by $\frakr^{*} \colon
C_{0}((\frakH G)^{0}) \to {\cal L}(L^{2}(\frakH G))$ the
representation given by $((\frakr^{*}h)f) (A) =
h(\frakr(A))f(A)$ for all $h \in C_{0}((\frakH G)^{0})$, $f
\in C_{c}(\frakH G)$, $A \in \frakH G$, and by
$s^{*}$ the pull-back of functions from $G^{0}$ to $G$ via
$s$.
\begin{lemma} \label{lemma:module} $ {\cal A}_{c}(G) =
  C_{c}(G)s^{*}(C_{0}(Y))$.
\end{lemma}
\begin{proof}
  Let $U,V \subseteq G$ be open Hausdorff $G$-sets and let
  $f\in C_{c}(U)$, $g \in C_{c}(V)$. Choose $h \in C_{c}(U)$
  such that $h|_{\supp f} = 1$.  Then $h^{*} \ast g, h
  \ast g \in C_{c}(G)$, 
  \begin{gather*}
    \begin{aligned}
      (h^{*} \ast g)(s(x)) &= g(x) \text{ for all } x \in
      \supp f, & f(x)g(x) &= f(x) (h^{*} \ast g)(s(x))
      \text{ for all } x \in G,
    \end{aligned}
\\
\begin{aligned}
g(s(x))  &=   (h \ast g)(x) \text{ for all } x \in \supp f, &
  f(x)g(s(x)) &= f(x)(h \ast g)(x) \text{ for all } x \in G,
\end{aligned}
  \end{gather*}
  and hence $fg = f s^{*}((h^{*} \ast g)|_{G^{0}})$ and
  $fs^{*}(g|_{G^{0}}) = f(h \ast g)$.  Using a partition of
  unity argument and the fact that $G$ is \'etale, we can
  conclude that $C_{c}(G)C_{c}(G) =
  C_{c}(G)s^{*}(C_{0}(Y))$. 
  By induction, we find that $C_{c}(G)^{n} =
  C_{c}(G)^{n-1}s^{*}(C_{0}(Y)) = \cdots =
  C_{c}(G)s^{*}(C_{0}(Y))$ for each $n \in \mathbb{N}$ and
  therefore ${\cal A}_{c}(G) =
  C_{c}(G)s^{*}(C_{0}(Y))$. 
\end{proof}
\begin{proposition}\label{proposition:module}
  \begin{enumerate}
  \item There exists an isomorphism $\Phi \colon L^{2}(G)
    \to L^{2}(\frakH G)$ of Hilbert $C^{*}$-modules such
    that $\Phi(f \odot a) =fs^{*}(a)$ for all $f \in
    C_{c}(G) \subseteq C_{c}(\frakH G)$,
    $a \in C_{0}(Y)$.
  \item The $C^{*}$-algebra generated by
    $\Ad_{\Phi}(C^{*}_{r}(G))$ and
    $\frakr^{*}(C_{0}((\frakH G)^{0}))$ is
  $C^{*}_{r}(\frakH G)$.
  \end{enumerate}
\end{proposition}
\begin{proof}
  Both assertions follow from  Lemma \ref{lemma:module} and
  the straightforward relations
  \begin{align*}
    \langle fs^{*}(a)|gs^{*}(b)\rangle_{L^{2}(\frakH G)} &=
    \langle f \odot a|g \odot b\rangle_{L^{2}(G)} &&
    \text{for all } f,g \in
    C_{c}(G), a,b \in C_{0}(Y), \\
   \Ad_{\Phi}(L_{G}(f))\frakr^{*}(h) &= L_{\frakH
    G}(fs^{*}(h))  && \text{for all }
  f \in C_{c}(G), h \in C_{0}(Y). \qedhere
  \end{align*}
\end{proof}

\end{document}